\documentclass[12pt]{article}
\usepackage{graphicx}
\usepackage{amssymb}
\usepackage{amsmath}
\usepackage{amsfonts}
\usepackage{amsthm}
\usepackage{accents}
\usepackage[english]{babel}
\usepackage[latin1,ansinew]{inputenc}
\usepackage{a4wide}
\usepackage{color}
\newcommand{\be}{\begin{eqnarray}}
\newcommand{\ee}{\end{eqnarray}}
\newtheorem{theo}{Theorem}%[section]
\newtheorem{prop}{Proposition}

\newtheorem{lemma}{Lemma}

\newcommand{\R}{\mathbb R}

\begin{document}
\date{June 13, 2023}

\title{Formation of Singularities in Solutions to \\ Ruggeri's Hyperbolic Navier-Stokes Equations}
\author{Heinrich Freist\"uhler}
\maketitle
A most prominent feature of hyperbolic systems of \emph{conservation laws} 
\be
U_t+F(U)_x=0, 
\ee 
$F:\mathcal U\to\R^n, \mathcal U\subset\R^n$,
is the property that many smooth solutions to them develop singularities in finite time. Such solutions exist
whenever the system has a genuinely nonlinear mode, i.e., there is a smooth (eigenvalue,eigenvector) pair 
$(\lambda,r):\mathcal U_*\to\R\times\R^n \setminus\{0\}$,
\be\label{mode}
(DF(U)-\lambda(U))r(U)=0,\quad U\in \mathcal U_*,
\ee
defined on an open neighborhood $\mathcal U_*\subset\mathcal U$ 
of some state
$U_*\in\mathcal U$ with  
\be\label{GNL}
(r(U)\cdot\nabla\lambda(U))|_{U=U_*}\neq 0.
\ee    
The condition \eqref{GNL} of \emph{genuine nonlinearity} goes back to Lax \cite{L57}, and first general results on 
singularity formation to him \cite{LSing}, John \cite{J}, and Liu \cite{Liu}. 

For systems of \emph{balance laws} 
\be
U_t+F(U)_x=G(U), 
\ee 
with source $G:\mathcal U\to\R^n$, one might wonder whether the same holds notably if the reference state is an 
equilibrium, i.e., satisfies
\be\label{EL}
G(U_*)=0.
\ee
In fact, typical physically interesting sources $G$ are \emph{damping} in such a way that arbitrary smooth data $U_0:\R\to \mathcal U$ for which 
an appropriate norm of $U_0(\cdot)-U_*$ is sufficiently small lead to solutions that remain smooth for all times $t>0$
\cite{D,KY}. However, B\"arlin has recently shown
\begin{prop}[\cite{B}]
Under the sole assumptions  \eqref{EL} and \eqref{GNL} there exist, for any Euclidean neighborhood $\mathcal U_*$ of $U_*$,
smooth data $U_0:\R\to \mathcal U_*$ such that the associated regular solution, while its values stay in $\mathcal U_*$, 
ceases to be differentiable after finite time.   
\end{prop}
In the present note we consider the hyperbolic compressible Navier-Stokes equations 
\be\label{EOM}
\begin{aligned}
\rho_t+(\rho u)_x&=0\\
(\rho u)_t+(\rho u^2 +p(\rho,\theta)+\sigma)_x&=0\\
(\rho (e(\rho,\theta)+|u|^2/2))_t
+((\rho(e(\rho,\theta)+|u|^2/2)+p(\rho,\theta)+\sigma)u)_x&=0\\
\epsilon\left((\rho\sigma/\theta)_t+
(\rho u \sigma/\theta)_x\right) 
+ u_x&=-\sigma/\eta.
\end{aligned}
\ee
with $p=R\theta\rho, e=c\theta$, and show the following consequence of Proposition 1. 
\begin{theo}
For any reference state $(\rho_*,u_*,\theta_*,\sigma_*)$ with $\rho_*,\theta_*>0$ and $\sigma_*=0$
and any $k>0$, there exist $C^\infty$ data $(\rho_0,u_0,\theta_0,\sigma_0)$ such that with them, the associated smooth
solution $(\rho,u,\theta,\sigma)$ to \eqref{EOM} assumes values in the ball of radius $k$ around $(\rho_*,u_*,\theta_*,\sigma_*)$
while it ceases to be differentiable after a finite time. 
\end{theo}
In equations \eqref{EOM}, $\rho,\theta,u,\sigma$ denote the local (extended) density, temperature, velocity, and stress
of the fluid 
%whose pressure and specific internal energy are given by $p=R\theta\rho$ and $e=c_v\theta$, while
where 
$\eta>0$ and $\epsilon>0$ are the coefficient of viscosity and a retardation parameter.    
These equations %\eqref{EOM} 
are the obvious no-heat-conduction variant of Ruggeri's 1983  
``symmetric-hyperbolic system of conservative equations for a viscous 
heat conducting fluid'' \cite{R83} and according to \cite{F22}, as a principal subsystem (in the sense of \cite{MR})
of Ruggeri's full model, a symmetric-hyperbolic system of Godunov-Boillat type \cite{Go,Bo}.

In view of B\"arlin's result, Theorem 1 is a consequence of the following fact.
\begin{lemma}
At every equilibrium state $(\rho_*,u_*,\theta_*,\sigma_*)$ with $\rho_*,\theta_*>0$ and $\sigma_*=0$, system
\eqref{EOM} has a mode \eqref{mode} which is genuinely nonlinear in the sense of \eqref{GNL}. 
\end{lemma}
\begin{proof} 
We view \eqref{EOM} as 
$$
A^0(V)V_t+A^1(V)V_x=\tilde G(V)
$$
with $V=(\rho,u,\theta,\sigma)$ 
and look for values of 
$$
\mu=u-\lambda
$$
that make 
$$
M(V,\mu)=-\lambda A_0(V)+A_1(V) 
=
\begin{pmatrix}
\mu&\rho&0&0\\
p_\rho&\rho\mu&p_\theta&1\\
-\theta p_\theta\mu/\rho &\sigma&\rho e_\theta\mu&0\\
\epsilon\sigma\mu/\theta&1+\epsilon\rho\sigma/\theta
&-\epsilon\sigma\rho\mu/\theta^2&\epsilon\rho\mu/\theta
\end{pmatrix}
$$
singular. From
$$
\det M(V,\mu)=\epsilon\frac\rho\theta\mu^2\bigg(\mu^2\rho^2e_\theta-(\sigma p_\theta+\rho^2p_\rho e_\theta+\theta p_\theta^2)\bigg)
-\mu^2\bigg(\rho e_\theta  %+2\epsilon\frac{\rho^2}\theta e_\theta\sigma
+\epsilon\frac\rho\theta p_\theta\sigma+\epsilon\frac\rho{\theta^2}\sigma^2\bigg),
$$
we find that the nontrivial speeds of the left-hand side operator satisfy
\be\label{exprmu}
\mu^2\!=\!p_\rho
+\frac\theta{\epsilon\rho^2}
+\frac1{\rho^2e_\theta}
\bigg[\theta p_\theta^2+2p_\theta \sigma+\frac{\sigma^2}\theta\bigg]
\!=\!R\theta+\frac\theta{\epsilon\rho^2}+\frac1{\rho^2c}
\bigg[
R^2\rho^2\theta+2R\rho\sigma+\frac{\sigma^2}\theta
\bigg].
\ee
To determine modes at equilibrium, we set $\sigma =0$ and   
correspondingly look for $r\in\R^4\setminus\{0\}$ in the kernel of 
$$
M_0(V,\mu)=
\begin{pmatrix}
\mu&\rho&0&0\\
p_\rho&\rho\mu&p_\theta&1\\
-\theta p_\theta\mu/\rho &0&\rho e_\theta\mu&0\\
0&1&0&\epsilon\rho \mu/\theta
\end{pmatrix}
$$
with, by assumption, $\mu\neq 0$.\footnote{The double mode associated with $\mu=0$ is linearly degenerate.} 
A quick computation yields
\be\label{evec}
r\sim\left(\epsilon\frac{\rho^2}\theta,-\epsilon\frac\rho\theta\mu,
\epsilon\frac{p_\theta}{e_\theta},1\right)
\ee
with $\mu$ satisfying 
\be\label{exprmu0}
\mu^2=p_\rho+\frac{\theta p_\theta^2}{\rho^2e_\theta}+\frac\theta{\epsilon\rho^2}
\ee
(which latter indeed is \eqref{exprmu} for $\sigma=0$).

While trivially 
$$
\frac{\partial\lambda}{\partial u}=1,
$$
we compute, at $\sigma=0$,
$$
\begin{aligned}
-2\mu\frac{\partial\lambda}{\partial\rho}=\frac{\partial \mu^2}{\partial\rho}&=-\frac{2\theta}{\epsilon\rho^3}\\
-2\mu\frac{\partial\lambda}{\partial\theta}=\frac{\partial \mu^2}{\partial\theta}&=R+\frac{R^2}c+\frac1{\epsilon\rho^2}\\
-2\mu\frac{\partial\lambda}{\partial\sigma}=\frac{\partial \mu^2}{\partial\sigma}&=\frac{2R}{\rho c}
\end{aligned}
$$
and thus 
$$
\begin{aligned}
 -2\mu r\cdot\nabla\lambda
&=\frac{\partial\mu^2}{\partial\rho}\frac{\epsilon\rho^2}\theta
-2\mu\frac{\partial\lambda}{\partial u}\bigg(\!\!-\epsilon\frac\rho\theta\mu\bigg)
+\frac{\partial\mu^2}{\partial\theta}\frac{\epsilon R\rho}c
+\frac{\partial\mu^2}{\partial\sigma}\\
&>-\frac2\rho+2\epsilon\frac\rho\theta\mu^2
\\
&>0.
\end{aligned}
$$
This implies \eqref{GNL} and thus the assertion.
\end{proof}
{\bf Remark.} 
Analoga of Lemma 1 and Theorem 1 hold for the
isothermal version \cite{F22}
\be\label{EOMisoth}
\begin{aligned}
\rho_t+(\rho u)_x&=0\\
(\rho u)_t+(\rho u^2 +R\rho+\sigma)_x&=0\\
\epsilon\left((\rho\sigma)_t+
(\rho u \sigma)_x\right) 
+ u_x&=-\sigma/\eta.
\end{aligned}
\ee
We next consider the non-isothermal version \emph{with} heat conduction \cite{R83,F22}
\be\label{EOMwith}
\begin{aligned}
\rho_t+(\rho u)_x&=0\\
(\rho u)_t+(\rho u^2 +p(\rho,\theta)+\sigma)_x&=0\\
(\rho (e(\rho,\theta)+|u|^2/2))_t
+((\rho(e(\rho,\theta)+|u|^2/2)+p(\rho,\theta)+\sigma)u+q)_x&=0\\
\epsilon\left((\rho\sigma/\theta)_t+
(\rho u \sigma/\theta)_x\right) 
+ u_x&=-\sigma/\eta\\
\delta((\rho q/\theta^2)_t+(\rho u q/ \theta^2)_x+\theta_x&=-q/\chi,
\end{aligned}
\ee
where $q$ is the heat flux and $\chi, \delta>0$ denote the heat conductivity and another retardation 
parameter, and show 
\begin{theo}
For any given value of $\theta_*>0$ there exist $\rho_*>0$ such that for 
the reference state $(\rho_*,u_*,\theta_*,\sigma_*,q_*)$ with $\sigma_*=q_*=0$ and any $k>0$, there exist 
$C^\infty$ data $(\rho_0,u_0,\theta_0,\sigma_0,q_0)$ such that with them, the associated smooth
solution $(\rho,u,\theta,\sigma,q)$ to \eqref{EOMwith} assumes values in the ball of radius 
$k$ around $(\rho_*,u_*,\theta_*,\sigma_*,q_*)$
while it ceases to be differentiable after a finite time. 
\end{theo}
In view of B\"arlin's result, Theorem 2 is a consequence of the following fact.
\begin{lemma}
For any $\theta_*>0$, there exist $\rho_*>0$ such that  \eqref{EOMwith} has a 
genuinely nonlinear mode at the reference state $(\rho_*,u_*,\theta_*,0,0)$.   
\end{lemma}
\begin{proof}
Genuine nonlinearity is invariant under coordinate changes.
We will use the equations in Lagrangian coordinates, 
\be
\begin{aligned}
\tau_t-u_y&=0\\
u_t+(p+\sigma)_y&=0\\
(e+|u|^2/2)_t+((p+\sigma)u+q)_x&=0\\
\epsilon(\sigma/\theta)_t 
+ u_y&=-\sigma/(\rho\eta)\\
\delta (q/\theta^2)_t 
+\theta_x&=-q/(\rho\chi) 
\end{aligned}
\ee
in $\tau,u,\theta,\sigma,q$.\\
(Lagrangian) characteristic speeds are those values of $\lambda$ which annihilate the determinant of 
$$
\begin{pmatrix}
-\lambda&-1&0&0&0\\
p_\tau&-\lambda&p_\theta&1&0\\
up_\tau-\lambda e_\tau&p+\sigma-\lambda u  &u p_\theta-\lambda e_\theta&u&1\\
0&1&\epsilon\sigma\frac 1{\theta^2}\lambda&-\epsilon\frac1\theta\lambda&0\\
0&0&2\delta q\frac1{\theta^3}\lambda+1&0&-\delta\frac1{\theta^2}\lambda
\end{pmatrix}
$$
or, equivalently, of 
$$
M(\tau,\theta,\sigma,q;\lambda)
=
\begin{pmatrix}
-\lambda&-1&0&0&0\\
p_\tau&-\lambda&p_\theta&1&0\\
-\lambda e_\tau&p+\sigma&-\lambda e_\theta&0&1\\
0&1&\epsilon\sigma\frac 1{\theta^2}\lambda&-\epsilon\frac1\theta\lambda&0\\
0&0&2\delta q\frac1{\theta^3}\lambda+1&0&-\delta\frac1{\theta^2}\lambda
\end{pmatrix}.
$$
At equilibrium, this matrix is 
$$
M_0(\lambda)
= M(\tau,\theta,0,0;\lambda)
=
\begin{pmatrix}
-\lambda&-1&0&0&0\\
p_\tau&-\lambda&p_\theta&1&0\\
-\lambda e_\tau&p&-\lambda e_\theta&0&1\\
0&1&0&-\epsilon\frac1\theta\lambda&0\\
0&0&1&0&-\delta\frac1{\theta^2}\lambda
\end{pmatrix},
$$
and has
$$
\det M_0(\lambda)
= 
-\frac\epsilon\theta\lambda\pi_0(\lambda)
$$
with
$$
\pi_0(\lambda^2)
=
\left\{
e_\theta\frac\delta{\theta^2}
\lambda^2[\lambda^2-\lambda_{**}^2]-[\lambda-\lambda_*^2]\}
\right\}
$$
where
$$
\lambda_*^2=-p_\tau+\frac\theta\epsilon,
\quad
\lambda^2_{**}=\lambda_*^2+\frac{\theta p_\theta^2}{e_\theta}.
$$
As 
$$
\pi_0(0)>0,\ \pi_0(\lambda_*^2)<0,\ \pi_0(\lambda_{**}^2)<0,\ \pi_0(\infty)>0,
$$
the quadratic polynomial $\pi_0$ has two (concretely given) 
real zeros $\lambda_-^2,\lambda_+^2$ 
with
\be\label{lambdaordering}
0<\lambda_-^2<\lambda_*^2<\lambda_{**}^2<\lambda_+^2. 
\ee
For any speed $\lambda\neq 0$, the kernel of $M_0(\lambda)$ 
is spanned by 
$$
r\sim\left(-\frac\epsilon\theta,\frac\epsilon\theta\lambda,
\frac\epsilon{\theta p_\theta}(\lambda^2-\lambda_*^2),1,
\frac{\epsilon\theta}{\delta p_\theta}\frac{\lambda^2-\lambda_*^2}\lambda
\right).
$$
At equilibrium, all four nontrivial speeds $-\lambda_+,-\lambda_-,\lambda_-,\lambda_+$ (like the 
trivial one)
are simple zeros, so near equilibrium they are smooth functions $\lambda(\tau,\theta,\sigma,q)$
that solve   
$$
\Pi(\tau,\theta,\sigma,q;\lambda(\tau,\theta,\sigma,q))
\equiv
\det(A(\tau,\theta,\sigma,q;\lambda(\tau,\theta,\sigma,q) )
=0.
$$
As 
$$
\frac{\partial\Pi}{\partial(\tau,\theta,\sigma,q)}
+
\frac{\partial\Pi}{\partial\lambda}\frac{\partial\lambda}{\partial(\tau,\theta,\sigma,q)}
=0
\quad\text{and}\quad
\frac{\partial\Pi}{\partial\lambda}(\tau,\theta,0,0,\lambda(\tau,\theta,0,0))\neq 0,
$$
an inequality  
\be\label{gnl}
N:=-\frac\epsilon\theta \frac{\partial\Pi}{\partial\tau}
+\frac\epsilon{\theta p_\theta}(\lambda^2-\lambda_*^2) \frac{\partial\Pi}{\partial\theta}
+
\frac{\partial\Pi}{\partial\sigma}
+
\frac{\epsilon\theta}{\delta p_\theta}\frac{\lambda^2-\lambda_*^2}\lambda
\frac{\partial\Pi}{\partial q}\neq 0
\ee
will thus imply the genuine-nonlinearity condition \eqref{GNL}.
 
Now, simple computations show that at equilibrium,
\be
\begin{aligned}
-\frac\epsilon\theta\frac{\partial\Pi}{\partial\tau}
&=
\frac{2\epsilon^2\delta R(R+c)}{\tau^3\theta^4}\lambda(\lambda^2-\lambda_\tau^2)\\
\frac\epsilon{\theta p_\theta}(\lambda^2-\lambda_*^2) \frac{\partial\Pi}{\partial\theta}
&=
\frac{\epsilon^2\delta c\tau}{R\theta}\lambda^3(\lambda^2-\lambda_*^2)(\lambda^2-\lambda_\theta^2) \\
\frac{\partial\Pi}{\partial\sigma}&=\frac{\epsilon\delta}{\theta^4}\lambda^3(p+e_\tau-\theta p_\theta)=0\\   
\frac{\epsilon\theta}{\delta p_\theta}\frac{\lambda^2-\lambda_*^2}\lambda
\frac{\partial\Pi}{\partial q}&=\frac{2\epsilon^2\tau}{R\theta^3}\lambda(\lambda^2-\lambda_*^2)^2
\end{aligned}
\ee
and thus 
\be\label{Nis}
N=
\alpha_\tau\lambda(\lambda^2-\lambda_\tau^2)
+\alpha_\theta\lambda^3(\lambda^2-\lambda_*^2)(\lambda^2-\lambda_\theta^2)
+\alpha_q\lambda(\lambda^2-\lambda_*^2)^2
\ee
with 
$$
\lambda_\tau^2=\frac{\theta^2}{\delta(R+c)},
\quad 
\lambda_\theta^2=\frac{\theta^2}{\delta c}+\frac{2\theta}\epsilon
$$
and 
\be\label{alphapositive}
\alpha_\tau,\alpha_\theta,\alpha_q>0.
\ee
As 
$$
\lambda_*^2=\frac{R\theta}{\tau^2}+\frac\theta\epsilon,
%\quad
%\lambda^2_{**}=\frac{R\theta}{\tau^2}\bigg(1+\frac Rc\bigg)+\frac\theta\epsilon.
$$
we see that by choosing $\tau>0$ sufficiently small, we have
\be\label{lambdaforsmalltau}
\lambda_\tau^2,\lambda_\theta^2<\lambda_*^2.
\ee
From \eqref{lambdaordering}, \eqref{lambdaforsmalltau}, \eqref{Nis}, and \eqref{alphapositive} 
we find that for the fast mode,
$$
\lambda^2=\lambda_+^2,
$$
the nonlinearity coefficient satisfies $N\neq 0$. 
\end{proof}
{\bf Remarks.} 
(i) Geometrically, genuine nonlinearity \eqref{GNL} is of course generic on open sets in state space $\mathcal U$, 
while, in the absence of (e.g., translational or rotational) symmetries, \emph{linear degeneracy} 
(i.\ e., $r(U)\cdot\nabla \lambda(U)=0$) can generically occur only 
along hypersurfaces (of ``inflection points''). It would be interesting to know whether linear 
degeneracy does occur for this system, notably at equilibrium, for the fast mode, $\pm\lambda_+$, or the slow mode, 
$\pm\lambda_-$.    \\
(ii) Racke has announced \cite{R23} a different proof for the formation of singularities
in a similar system.

\end{document}